\documentclass[final,3p,times]{elsarticle}

\usepackage{amssymb,amsthm,amsmath}
\usepackage{algorithmic,algorithm}

\usepackage{graphics}
\usepackage{enumerate}
\usepackage{color}
\usepackage{mathtools}
\usepackage[normalem]{ulem}
\usepackage{scalerel}

\newcommand{\NEW}[1]{\begingroup\color{black}#1\endgroup}
\newtheorem{thm}{Theorem}
\newtheorem{defi}[thm]{Definition}
\newtheorem{lemma}[thm]{Lemma}
\newtheorem{cor}[thm]{Corollary}

\DeclareRobustCommand{\Ro}{\overset{0}{\sim}}
\DeclareRobustCommand{\Rl}{\overset{1}{\sim}}
\DeclareRobustCommand{\Rk}{\overset{k}{\sim}}
\DeclareRobustCommand{\Rt}{\overset{2}{\sim}}

\DeclareRobustCommand{\Rtd}{\overset{2}{\rightharpoonup}}
\DeclareRobustCommand{\Rkd}{\overset{k}{\rightharpoonup}}
\DeclareRobustCommand{\lca}[1]{\mathop{lca}(#1)}

\newcommand{\rthin}{\mathrel{\mathrel{\ooalign{\hss\raisebox{-0.17ex}{$\sim$}\hss\cr\hss\raisebox{0.720ex}{\scalebox{0.75}{$\bullet$}}\hss}}}}

\begin{document}


\begin{frontmatter}

\title{Exact-$2$-Relation Graphs}

\author[WUH]{Yangjing Long}
\ead{yangjing@mail.ccnu.edu.cn}

\author[LEI,MIS,VIE,UNAC,SFI]{Peter F. Stadler}
\ead{stadler@bioinf.uni-leipzig.de}

\address[WUH]{School of Mathematics and Statistics, Center China Normal
  University, No. 152, Luoyu Road, Wuhan, Hubei, P.\ R.\ China}

\address[LEI]{Bioinformatics Group, Department of Computer Science;
  Interdisciplinary Center for Bioinformatics; German Centre for
  Integrative Biodiversity Research (iDiv) Halle-Jena-Leipzig; Competence
  Center for Scalable Data Services and Solutions Dresden-Leipzig; Leipzig
  Research Center for Civilization Diseases; and Centre for Biotechnology
  and Biomedicine, University of Leipzig, H{\"a}rtelstra{\ss}e 16-18,
  D-04107 Leipzig, Germany}

\address[MIS]{Max Planck Institute for Mathematics in the Sciences,
  Inselstra{\ss}e 22, D-04103 Leipzig, Germany}

\address[VIE]{Institute for Theoretical Chemistry, University of Vienna,
  W{\"a}hringerstra{\ss}e 17, A-1090 Wien, Austria}

\address[UNAC]{Facultad de Ciencias, Universidad Nacional de Colombia,
  Sede Bogot{\'a}, Colombia}

\address[SFI]{The Santa Fe Institute, 1399 Hyde Park Rd., Santa Fe, NM
  87501, United States}

\begin{abstract}
  Pairwise compatibility graphs (PCGs) with non-negative integer edge
  weights recently have been used to describe rare evolutionary events and
  scenarios with horizontal gene transfer. Here we consider the case that
  vertices are separated by exactly two discrete events: Given a tree $T$
  with leaf set $L$ and edge-weights $\lambda: E(T)\to\mathbb{N}_0$, the
  non-negative integer pairwise compatibility graph
  $\textrm{nniPCG}(T,\lambda,2,2)$ has vertex set $L$ and $xy$ is an edge
  whenever the sum of the non-negative integer weights along the unique
  path from $x$ to $y$ in $T$ equals $2$.  A graph $G$ has a representation
  as $\textrm{nniPCG}(T,\lambda,2,2)$ if and only if its point-determining
  quotient $G/\!\rthin$ is a block graph, where two vertices are in relation
  $\rthin$ if they have the same neighborhood in $G$. If $G$ is of this
  type, a labeled tree $(T,\lambda)$ explaining $G$ can be constructed
  efficiently.  In addition, we consider an oriented version of this class
  of \NEW{graphs}.
\end{abstract}


\end{frontmatter}

\section{Introduction}

Consider a tree $T$ with leaf set $V$ and a non-negative edge-weight
function $\ell:E(T)\to\mathbb{R}^+_0$. Denote by $\mathcal{P}(x,y)$ the 
unique path between $x$ and $y$ in $T$. The canonical distance function  
$d_{T,\ell}:V\times V\to\mathbb{R}^+_0$ is then defined by 
\begin{equation}
  d_{T,\ell}(x,y) = \sum_{e\in\mathcal{P}(x,y)} \ell(e)
  \label{eq:additiveM}
\end{equation} 
This definition is the starting point for mathematical phylogenetics, which
is centered around finite additive metric spaces and their generalizations
\cite{Semple:03,Dress:12}. It also serves as a basis for defining a large
class of graphs that in the recent past has received considerable
attention. The \emph{pairwise compatibility graphs}
$G=\mathrm{PCG}(T,\ell,d_{\min},d_{\max})$ has vertex set $V$ and edges
\begin{equation} 
  uv \in E(G) \quad\textrm{if and only if}\qquad 
  d_{\min}\leq d_{T,w}(u,v) \leq d_{\max}
\end{equation}
Originally introduced in the context of phylogenetics \cite{Kearney:03},
they have received considerable interest in the last years, see
\cite{Yanhaona:09,PCGsurvey} and the references therein, as well as
\cite{Hossain:17,Baiocchi:18,Baiocchi:19}. A further generalization of
``multi-interval'' PCGs in explored in \cite{Ahmed:17}.  In the setting of
PCGs and most phylogenetic applications one usually stipulates that
$\ell(e)>0$, measuring e.g.\ the time between two distinct events
associated with adjacent vertices of $T$. A class of graphs that is
conceptually closely are the exact $k$-leaf powers \cite{Brandstaedt:10},
for which $\lambda(e)=1$ for all edges of $T$ and $d_{\min}=d_{\max}=k$.

In an alternative interpretation, $\ell(e)$ models the number of
\emph{discrete} evolutionary events along an edge of $e$ of $T$.  This is
of interest in particular in the context of so-called \emph{rare genomic
  changes} such as the gene or loss of a particular gene of gene family or
a particular genomic rearragement \cite{Rokas:00a}. Some of these convey
phylogenetic information that is (nearly) free of homoplasy, i.e., the
independent occurrance in independent lineages. \NEW{Examples of such rare
  events are the emergence of novel microRNA families \cite{Tarver:13} or
  rearrangements of the genomic gene order \cite{Luo:12}. Since such events
  are very unlikely to have occurred more than once in the same manner,
  they identify phylogenetic groupings that share such an innovation with
  very little ambiguity. This provides information to resolve also parts
  the phylogenetic tree where classical, sequence-based methods fail
  \cite{Waegele:DMP}.}  In this context it is necessary to allow
$\ell(e)=0$ because the events of interest are by definition so rare that
not all taxa will be distinguished by them. In the same vein it important
that events are discrete and hence an integer-valued weight function
$\lambda:E(T)\to\mathbb{N}_0$. Both conditions on $\ell$ cause subtle but
important differences in comparison with the usual definition of PCGs that
requires non-zero edge length but otherwise allows arbitrary real
values. We will denote these ``non-negative integer pairwise compatibility
graph'' by $\textrm{nniPCG}$ to distinguish them from the better studied
class of PCGs with non-zero real-valued edge weights $\lambda$.

The special case in which two leaves $x$ and $y$ of $T$ are separated by a
single event, corresponding to graphs of the form
$\mathrm{nniPCG}(T,\lambda,1,1)$, was explored in \cite{Hellmuth:17q} as a
model of rare events in evolution. It turns out that this graph class
coincides with the forests. The graphs
$\mathrm{nniPCG}(T,\lambda,1,\infty)$ requiring at least one event along
the path between two leaves also have a very simple structure: they are
exactly the complete multipartite graphs \cite{Hellmuth:18x}.

Considering a rooted tree $\vec{T}$ instead of an unrooted tree $T$, it is
natural to consider the digraphs with edges $(x,y)$ whenever a certain
number of events occured between the last common ancestor $\lca{x,y}$ and
$y$. This construction appears naturally in the context of horizontal gene
transfer (HGT), where one asks whether $\lca{x,y}$ and $y$ are separated by
at least one HGT event\footnote{HGT refer to the import of gene from an
  unrelated species e.g.\ through an infection, ingestion, acquisition via
  a plasmid.} This gives rise to the class of Fitch graphs
\cite{Geiss:17a}, which form a subclass of di-cographs introduced by
\cite{CP-06}. Their underlying undirected graphs are exactly the
$\mathrm{nniPCG}(T,\lambda,1,\infty)$, i.e., the complete multipartite
graphs.

A related construction requires a certain number of events between
$\lca{x,y}$ and $y$ \emph{and} excludes all events between $\lca{x,y}$ and
$x$. This class of graphs appears naturally when events are directed, i.e.,
when it is (in general) no possible to revert the effect of an operation in
a single step. Probably the best studied type of single genomic events of
this type are so-called Tandem-Duplication-Random-Loss events, during which
a genomic interval is duplicated and then one of the two copies of each
gene is lost at random \cite{Chaudhuri:06}. The antisymmetric digraphs
obtained by single events are characterized in \cite{Hellmuth:17q}.

Our interest in the graphs $\mathrm{nniPCG}(T,\lambda,k,k)$ for general
$k\ge 1$ also stems from rare-event phylogenetic data. Since we assume an
underlying tree, the distance matrix $d_T$ is additive and its entries are
small non-negative integers. The fact that all edge lengths $\ell(e)$ are
also integers of course imposes additional constraints. As demonstrated
e.g.\ in the context of orthology assignment (a related problem with vertex
labeled trees for which the corresponding graphs turn out to be cographs),
graph editing can be employed to correct empirically estimated input graphs
\cite{Hellmuth:13a}. This approach requires, however, that constraint on
the graphs that can appear are known. In the case of rare-event
phylogenetics, we know that the graph with edge set $\{xy|d_T(x,y)=k\}$
must be a $\mathrm{nniPCG}(T,\lambda,k,k)$. In the rare-event scenario, the
number of pairs of nodes with $d_T(x,y)=k$ will quickly decrease with $k$,
so that the empirical input graphs will have few edges for larger values of
$k$ and thus rarely reveal obstructions.  Hence only small value of $k$ are
of practical value for detecting measurement errors in the data.  Since the
$\mathrm{nniPCG}(T,\lambda,1,1)$ are \NEW{forests}, the corresponding graph
editing problem amounts to identifying spanning \NEW{forests}, and possible
false positive events are edges in cycles. False negatives are not
detectable for $k=1$ since there are no non-tree graphs that would become
trees by inserting edges. They could be detected, however, as missing edges
in the empirical graph for $k=2$ compared to the most similar member of
$\mathrm{nniPCG}(T,\lambda,2,2)$. In this contribution, therefore, we are
interested in the characterization of the graphs
$\mathrm{nniPCG}(T,\lambda,2,2)$, in which edges correspond to exactly two
events between two leaves. \NEW{This graph class is very different from the
  exact-2-leaf power graphs, which are known to coincide with the disjoint
  unions of cliques \cite{Nishimura:02,Brandstaedt:10}.  In contrast, we
  shall see below that e.g.\ every path also has a representation as
  $\mathrm{nniPCG}(T,\lambda,2,2)$.}

This contribution is organized as follows: We first consider a few general
properties of the slightly more general \emph{exactly-$k$-relation} $\Rk$
before investigating for some small graphs and simple graph families
whether they can be respresented with respect to the
\emph{exactly-$2$-relation} $\Rt$ on the leaf set of some tree. Here, we
consider the case that $d_T(x,y)>0$, i.e., that all leaves are separated by
at least one event, and then relax this constraint and characterize the
entire graph class $\mathrm{nniPCG}(T,\lambda,2,2)$ for non-negative,
integer $\lambda$. Our main result is that these graphs are those whose
quotient with respect to the \NEW{false twin (R-thinness)} relation is a
block graph. We then consider the oriented version of the problem and give
characterization in terms of forbidden subgraphs.

\section{Simple Properties of the Exactly-$k$-Relation}
\label{sect:general} 

We shall see that the restriction to integer edge weights on the one hand,
and the admission of zero-weights on the other hand, make the graphs
$\mathrm{nniPCG}(T,\lambda,k,k)$ quite different from the exact-$k$-leaf
power graphs studies systematically in \cite{Brandstaedt:10}. \NEW{While it
  is true that every $\mathrm{PCG}(T,\lambda,d_{\min},d_{\max})$ with
  non-negative real weigths $\lambda$ and bounds $d_{\min}$ and $d_{\max}$
  also has a representation as
  $\mathrm{PCG}(T,\hat\lambda,\hat d_{\min},\hat d_{\max})$ with integer
  weights and bounds \cite[Lemma 2]{Calamoneri:13}, the restriction to
  integer weights clearly affects the definition of graph
  \emph{classes}. For instance, the PCG class with rational weights and
  bounds $d_{\min}=d_{\max}=1$ contains the
  $\mathrm{nniPCG}(T,\lambda,k,k)$ for all $k\in\mathbb{N}$.} Throughout
this contribution we use a notation that is inspired by related work in
mathematical phylogenetics.

\begin{defi} 
  Let $(T,\lambda)$ be an unrooted tree with leaf set $L$ and edge-labeling
  function $\lambda: E(T)\to\mathbb{N}_0$. For $x,y\in L$ we consider the
  \emph{exactly-$k$-relation $\Rk$} defined by \emph{$x \Rk y$} if the
  (unique) path $\mathbb(x,y)$ from $x$ to $y$ in $T$ satisfies
  $\sum_{e\in\mathcal{P}(x,y)} \lambda(e)=k$.  \newline Furthermore, we say
  $(T,\lambda)$ \emph{explains} a graph $G(L,E)$ \emph{(with respect to the
    exactly-$k$-relation)} if $\{x,y\}\in E$ if and only if $x\Rk y$.
  \label{def:exact-k}
\end{defi}

We consider unrooted instead of rooted tree since the distances $d_T(x,y)$
and thus the \emph{exactly-$k$-relation $\Rk$} contains no information on
position of root. In fact, it is well known
\cite{Buneman:71,SimoesPereira:69} that a metric $d$ of the form
(\ref{eq:additiveM}) uniquely defines an unrooted tree. Therefore, one can
only hope to reconstruct the unrooted tree $T$.

\begin{figure}
  \begin{center}
    \includegraphics[width=0.9\textwidth]{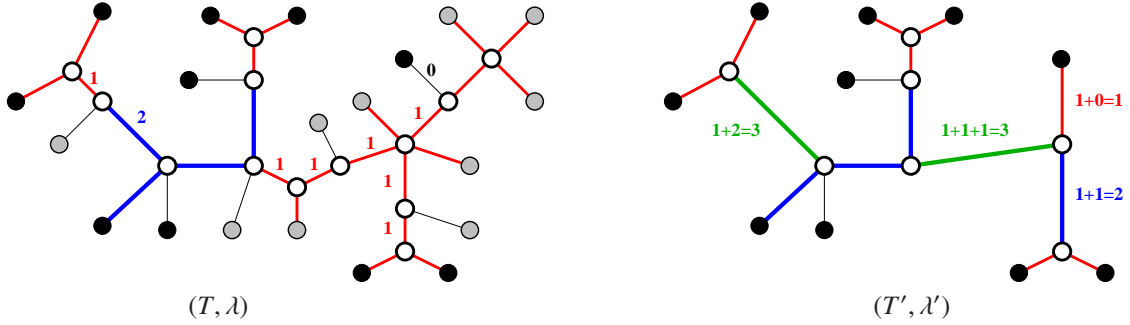}
    \par
    $(T,\lambda)$\hspace*{0.5\textwidth}$(T',\lambda')$
    \par
  \end{center}
  \caption{\NEW{Illustration of Definition~\ref{def:displayed}. The
      edge-labeled tree $(T',\lambda')$ on the r.h.s.\ is \emph{displayed}
      by $(T,\lambda)$. It is obtained as the restriction of $T$ to the
      non-gray vertices.}  Correponding vertices are shown in matching
    locations. All edges $e$ that are ``merged'' into single edges have
    their weights annotated.  Edges that remain unchanged or deleted are
    only shown in color (black for $\lambda(e)=0$, red for $\lambda(e)=2$,
    and blue of $\lambda(e)=2$) without displaying the weight explicitly.}
    \label{fig:displayed}
\end{figure}

\begin{defi}
  The edge-labeled tree $(T,\lambda)$ displays the edge labeled tree
  $(T',\lambda')$ if $(T',\lambda')$ can be obtained from $(T,\lambda)$ by
  first removing every edge and vertex from $(T,\lambda)$ that is not
  contained in a path connecting two leaves of $T'$, and then contracting
  every path $\mathcal{P}(u,v)$ in the remainder of $T$ that has only
  interior vertices of degree $2$ by a single edge $e'$ in $T'$ with label
  $\lambda'(e') = \sum_{e\in\mathcal{P}(u,v)} \lambda(e)$.
  \label{def:displayed}
\end{defi}
In particular, therefore it is sufficient to consider phylogenetic trees,
that is, trees $T$ in which every interior node $x\in V(T)\setminus L$ has
degree at least $3$. Fig.~\ref{fig:displayed} gives an example.
A simple, but important consequence of Definition \ref{def:exact-k} is the
following
\begin{lemma} 
  If $(T,\lambda)$ displays $(T',\lambda')$, $(T,\lambda)$ explains 
  $G(L,\Rk)$ and $(T',\lambda')$ explains $G'(L',\Rk)$.\\ Then
  $G'(L',\Rk)=G(L,\Rk)[L']$, the subgraph of $G(L,\Rk)$ induced by $L'$.
  \label{lem:hered}
\end{lemma}
\begin{proof}
  If $(T,\lambda)$ displays $(T',\lambda')$ then
  $\sum_{e\in\mathcal{P}_{T}(u,v)} \lambda(e)=
  \sum_{e\in\mathcal{P}_{T}(u,v)} \lambda'(e)$ for all $u,v\in
  L(T')\subseteq L(T)$, and thus we conclude that for all $u,v\in \in
  L(T')$, we have $u \Rk_{(T',\lambda')} v$ if and only if $u
  \Rk_{(T',\lambda')} v$, i.e., $G'(L',\Rk)$ is the subgraph of 
  $G(L,\Rk)$ induced by $L'$. 
\end{proof}
It follows that ``being explained with respect to the
exactly-$k$-relation'' is a hereditary graph property for all $k$.

We also note the following immediate consequence of the definition.
\begin{lemma} \label{lem:1tok} If $(T,\lambda)$ explains $G$ with respect
  to $\Rl$, \NEW{t}hen $(T,k\lambda)$ explains $G$ with respect to $\Rk$.
\label{lem:multiple}
\end{lemma}

\begin{lemma} 
  Let $G$ be a graph with connected components $G_i$, $i=1,\dots,N$. Then
  there is an edge-labeled tree $(T,\lambda)$ explaining $T$ with respect
  to $\Rk$ if and only if there are edge labeled trees $(T_i,\lambda_i)$
  explaining $G_i$ for all $i=1,\dots,N$.
  \label{lem:connected}
\end{lemma}
\begin{proof} 
  The condition is necessary because of heredity. In order to see
  sufficiency, we can construct $(T,\lambda)$ from the disjoint union of
  the $(T_i,\lambda_i)$ in the following way: first we arrange them as an
  arbitrary tree $\mathcal{T}$. Then we replace each $(K_2,\lambda(e)=k)$
  by $S_2$ with the two edges $e'$ and $e''$ labeled such that
  $\lambda(e')+\lambda(e'')=k$. Now choose for each tree $T_i$ an arbitrary
  inner vertex $x_i$ in $T_i\ne K_1$ and the unique vertex $x_i$ if
  $T_i=K_1$. Finally, we connect $x_i$ and $x_j$ by an edge $e_{ij}$ with
  $\lambda(e_{ij})=k+1$ if and only if $T_i$ and $T_j$ are adjacent in
  $\mathcal{T}$. To verify that $(T,\lambda)$ indeed explains $G$ we
  observe: (i) If $x$ and $y$ are leafs from different connected components
  of $G$, they are located in different subtrees $T_i$ and thus the path
  connecting them contains one of the edges label $k+1$, thus $x$ and $y$
  are not in relation $\Rk$.
\end{proof}
It is therefore sufficient to consider connected graphs. 

\begin{defi}
  An edge-labeled graph $(T,\lambda)$ is \emph{canonical} if $T$ is
  phylogenetic and $\lambda(e)\ne 0$ for all interior edges.
\end{defi}

\begin{lemma} 
\label{lem:canonical}
  Let $(\hat T,\hat\lambda)$ be the edge labeled tree obtained from 
  $(T,\lambda)$ by (1) replacing every path $\mathcal{P}(u,v)$ in $T$ whose
  interior vertices have degree $2$ by a single edge $e'$ in $T'$ with
  label $\lambda'(e') = \sum_{e\in\mathcal{P}(u,v)} \lambda(e)$ and (2)
  contracting every interior edge with $\lambda(e)=0$.  The tree
  $(\hat T,\hat\lambda)$ is uniquely defined, canonical, and explains the
  same graph as $(T,\lambda)$.
\end{lemma}
\begin{proof}
  The maximal paths with interior vertices of degree $2$ in $T$ are
  disjoint and thus can be treated independently. By construction, any such
  path $\mathcal{P}$ can also be stepwisely replaced by edges, eventually
  arriving at the same edge weight for the single edge that remains. Given
  $T$, the resulting tree $\hat T$ is therefore unique and contains no
  vertex of degree $2$. It is therefore phylogenetic. Since an interior
  edge with label $\lambda(e)=0$ does not contribute the total weight of
  any path that runs through it, it can be contracted without changing the
  total path weights between leaves. Thus $(T,\lambda)$ and
  $(\hat T,\hat\lambda)$ explain the same graph.
\end{proof}

Consider two leaves $x,y\in L$ in an edge-labeled tree $(T,\lambda)$ such
that $x\Ro y$, i.e., $d_T(x,y)=0$, and another leaf
$z\in L\setminus\{x,y\}$. The triangle inequalities
$d_T(x,z)\le d_T(x,y)+d_T(y,z)$ and $d_T(y,z)\le d_T(y,x)+d_T(x,z)$ implies
$d_T(x,z)=d_T(y,z)$. Thus $x$ and $y$ have the same neighbors in graph $G$
explained by $T$, i.e., $N_{G}(x)=N_G(y)$.
  
\begin{defi} 
  Let $G$ be a graph. For each $x\in V(G)$ denote by $N(x)$ the neighbors
  of $x$. Two vertices $x$ and $y$ are \NEW{\emph{false twins}},
  $x\rthin y$, if $N(x)=N(y)$.
\end{defi}
\NEW{In contrast, true twins, which play no role here, satisfy
  $N(x)\cup\{x\}=N(y)\cup\{y\}$. By definition, false twins $x\rthin y$ are
  non-adjacent, while true twins are always adjacent \cite{Bandelt:86}.}
  
\begin{figure}
  \begin{center}
    \includegraphics[width=0.9\textwidth]{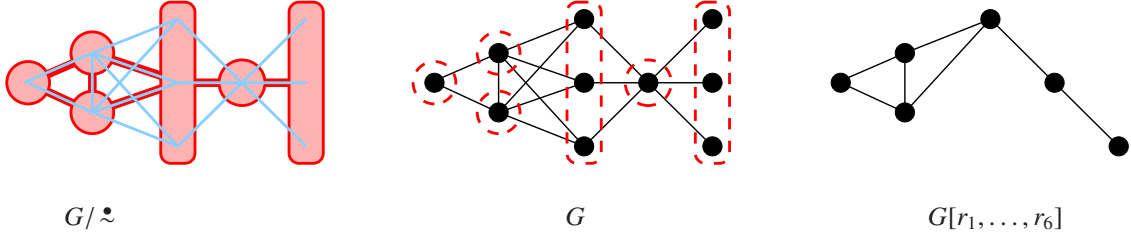}
  \end{center}
  \hfil  $G/\!\rthin$ \hfil \qquad\qquad\qquad\qquad\qquad\qquad  $G$ \hfil
   \qquad\qquad\qquad\qquad $G[r_1,\dots,r_6]$ \hfil
  \par 
  \caption{\NEW{The false twin classes are indicated by dashed outlines in
      the graph $G$ (middle) and form the vertices (in red) of the quotient
      graph $G/\!\protect\rthin$ on the right. For comparison, the orginal
      edges between the $\protect\rthin$ classes are shown.  The subgraph
      $G[r_1,r_2,r_3,r_4,r_5,r_6]$ of $G$ induced by one representative
      $r_i$ for each $\protect\rthin$-class is shown on the left. It is
      isomorphic to $G/\protect\rthin$ irrespective of the choice of the
      representative vertices $r_i$ for the $\protect\rthin$ classes.}}
  \label{fig:twin}
\end{figure}

The \NEW{false twin} (R-thinness) relation $\rthin$ has been well studied
in the literature, in particular in the context of graph products
\cite{Hammack:11}.  It is well known that $\rthin$ is an equivalence
relation, see e.g.\ \cite[sect.\ 8.2]{Hammack:11}. Its equivalence classes,
which we denote by $R_i$, $i=1,\dots,h$, are totally disconnected in $G$
because, by definition, $x\notin N(x)$. Denote by $G[r_1,r_2,\dots,r_h]$ be
the subgraph of $G$ induced by one arbitrarily chosen representative
$r_i\in R_i$ of each \NEW{false twin} class. Since for any $x\in R_i$ and
$y\in R_j$ we have $xy\in E(G)$ if and only $x'y'\in E(G)$ for all
$x'\in R_i$ and all $y'\in R_j$ we observe that $G[r_1,r_2,\dots,r_h]$ and
the quotient graph $G/\!\rthin$ are isomorphic. \NEW{An illustration is given
  in Fig.~\ref{fig:twin}.}

\begin{lemma} 
  \label{lem:Ro-thin}
  Let $(T,\lambda)$ be a canonical tree explaining a connected graph $G$
  with respect to $\Rk$, and let $W$ be a set of sibling leaves attached to
  the same parent $q$ \NEW{with} $\lambda(qw)=0$ for all $w\in W$. Then $W$
  is contained in a \NEW{false twin} class for the graph $G$ explained by
  $(T,\lambda)$ with respect to $\Rk$ for all $k>0$.
\end{lemma}
\begin{proof} 
  Consider a node $y\in L\setminus W$. Then the total weight of the path
  between $y$ and every $w\in W$ is the same. Furthermore, the total path
  weight between any two vertices in $w',w''\in W$ is $0$, i.e., there is
  no edge between $w'$ and $w''$. Thus $N(w')=N(w'')$, i.e., $w'\rthin w''$
  for all $w',w''\in W$. 
\end{proof}

A graph is called \emph{R-thin} \cite{McKenzie:71}, \emph{point determining
  graph} \cite{Sumner:73} or \emph{mating graph} \cite{Bull:89} if $\rthin$
is discrete, i.e., every \NEW{false twin} class consists of only a single
point. Clearly, $G/\!\rthin$ is R-thin. R-thin graphs have also been studied
from the point of view of combinatorial enumeration
\cite{Kilibarda:07,Gessel:11}. Algorithms for prime-factorization of
graphs, furthermore, often operate on $G/\!\rthin$, since R-thinness ensures
uniqueness of the factorization and allows for highly efficient algorithms
\cite{McKenzie:71,Imrich:98,Hammack:11}. Below we show that it also
suffices to consider $G/\!\rthin$, i.e., R-thin graphs, in our
setting. Indeed, a simple consequence of Lemma~\ref{lem:Ro-thin} is
\begin{cor} 
  \label{cor:Ro-thin}
  If $G$ is R-thin and $(T,\lambda)$ is a canonical tree explaining $G$ with
  respect to $\Rk$, then $\Ro$ is discrete.
\end{cor}

\begin{algorithm} 
  \caption{Compute $(T,\lambda)$ from $(T^*,\lambda^*)$ and \NEW{false twin}
    classes $R_i$ with representatives $r_i\in R_i$.}
\label{alg:blow-up}
\begin{algorithmic}[1]
  \REQUIRE $(T^*,\lambda^*)$, $(r_i,R_i)$ for $i=1,\dots,h$. 
  \FORALL {\NEW{false twin} classes $R_i$ with $|R_i|>1$}
     \STATE $q\leftarrow$ unique neighbor of leaf $r_i$ in $(T^*,\lambda^*)$
     \STATE remove $r_i$ from $(T^*,\lambda^*)$
     \IF { $\lambda^*(qr_i)\ne k/2$ } 
        \STATE insert all leaves $r\in R_i$ with edges 
          $qr$ and $\lambda(qr)=\lambda^*(qr_i)$
     \ELSE 
         \STATE insert a node $q'$ and the edge $qq'$ with 
           $\lambda(qq')=\lambda^*(qr_i)$ 
         \STATE insert all leaves $r\in R_i$ with edges 
           $q'r$ and $\lambda(q'r)=0$
     \ENDIF 
  \ENDFOR
\end{algorithmic}
  \NEW{For an illustrative example see Fig.~\ref{fig:alg1}.}
\end{algorithm} 

\begin{figure}[hbt]
  \begin{center}
    \includegraphics[width=0.95\textwidth]{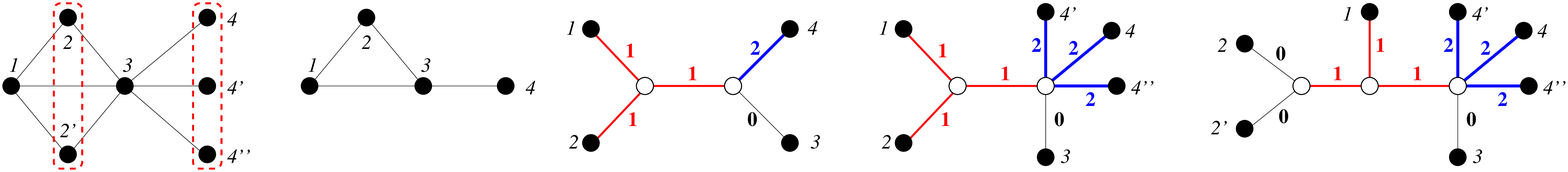}
  \end{center}
  \par
  \hfil $G$ \hfil $G/\!\protect\rthin$ \hfil $(T^*,\lambda^*)$ \hfil
  $(T_2,\lambda_2)$ \hfil $(T_{2,4},\lambda_{2,4})$ \hfil
  \par
  \caption{\NEW{The graph $G$ has two non-trivial false twin classes
      $R_2=\{2,2'\}$ and $R_4=\{4,4',4''\}$. Its R-thin quotient
      $G/\!\protect\rthin$ is explained w.r.t.\ $\Rt$ by the tree
      $(T^*,\lambda^*)$. The tree $(T_2,\lambda_2)$ is obtained from
      $(T^*,\lambda^*)$ using first alternative in Alg.~\ref{alg:blow-up}
      (line 5), amounting to replacing $2$ by $|R_2|$ leaves with the same
      parent.  In the second step, vertex $4$ is replaced by the subtree
      according to the second alternative in Alg.~\ref{alg:blow-up} (lines
      7 and 8).}  }
  \label{fig:alg1}
\end{figure}

\begin{thm}
  $G$ can be explained w.r.t.\ $\Rk$ if and only $G/\!\rthin$ can explained
  w.r.t.\ $\Rk$. If $(T^*,\lambda^*)$ is a canonical tree explaining
  $G/\!\rthin$, then a canonical tree $(T,\lambda)$ explaining $G$ is
  obtained by Algorithm~\ref{alg:blow-up}.
  \label{thm:thin} 
\end{thm}
\begin{proof}
  Since $G$ can be explained and $G/\!\rthin$ is an induced subgraph of $G$,
  $G/\!\rthin$ can be explained w.r.t.\ to $\Rk$ by a tree that we denote by
  $(T^*,\lambda^*)$. Let $r$ be the representative of the \NEW{false twin} class
  $R$ of $G$, and let $x\in R$. Insert $x$ into $(T^*,\lambda^*)$ are a
  sibling of $r$ and set $\lambda(x)=\lambda(r)=\lambda^*(r)$. Then $x$ and
  $r$ have the same total path weights to all other vertices. This remains
  true if each leaf $r$ in $(T^*,\lambda^*)$ is replaced in this manner by
  the set $R$ of sibling vertices with $r\in R$. Since no two vertices in
  $R$ are adjacent we require that $\lambda(r)+\lambda(x)\ne k$, i.e.,
  $\lambda^*(r)\ne k/2$. If this conditions is satisfied, then
  $(T,\lambda)$ explains $G$ with respect to $\Rk$.

  If $\lambda^*(qr)= k/2\ge 1$, Alg.~\ref{alg:blow-up} inserts an extra
  vertex $q'$ adjacent to $q$ with $\lambda(qq')=\lambda^*(qr)\ne0$. Since
  we assumed that $(T^*,\lambda^*)$ was canonical, $q$ has at least two
  more neighbors, i.e., the resulting tree is again canonical. Since $R$ is
  attached with edge weights $\lambda(q'r)=0$ we conclude that (i) the
  total path weight between $r'$ and $q$ is $k/2$ and (ii)
  $\lambda(r'q')+\lambda(q'r'')=0$ for all $r',r''\in R$, i.e., $r'$ and
  $r''$ are not adjacent in the graph explained by $(T,\lambda)$.  Hence
  $r'$ and $r''$ have the same neighbors and thus belong the same
  \NEW{false twin} class of $G$. Since the total path weights between all
  representatives of \NEW{false twin} classes are preserved by this
  construction, $(T,\lambda)$ indeed explains $G$ with respect to $\Rk$.
  We note, finally, that $(T,\lambda)$ is again canonical because $q'$ has
  at least three neighbors (the parent $q$ and at least members of $R$),
  and all interior edges the resulting tree have non-zero labels as long as
  $(T^*,\lambda^*)$ was canonical.
\end{proof}

\emph{From here on we will therefore assume the $(T,\lambda)$ is canonical,
  i.e., it has non-zero labels for all inner edges of $T$.}  It is
important to note, however, that we still need to consider zero weights on
the edges incident with leaves. \NEW{For instance, it not difficult to
  check that the graph $G/\!\rthin$ in Fig.~\ref{fig:alg1}, i.e., $K_3+e$,
  cannot be explained by a tree with only non-zero edge weights.}

\section{Graphs Explained w.r.t.\ $\Rt$}

We will first consider the special case of edge labelings with discrete
$\Ro$. In this case every interior vertex of $T$ is incident with at most
one zero-weight edge.

\begin{figure}[t]
  \begin{center}
      \includegraphics[width=1.0\textwidth]{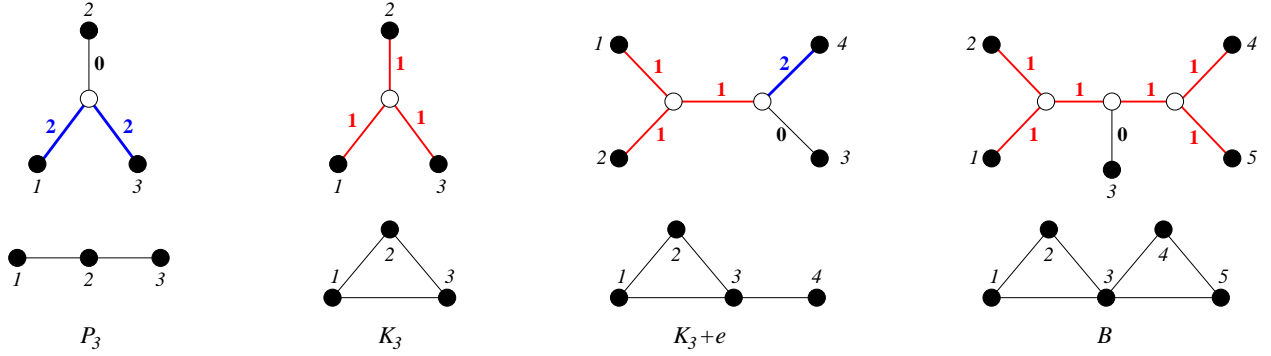}
  \end{center}
  \caption{The two connected graphs $P_3$ and $K_3$ on three vertices, the
    graph $K_3+e$ on four vertices, and bowtie graph $B$ on five vertices
    are explained by unique $\{0,1,2\}$-edge-labeled trees with respect to
    $\Rt$ and discrete relation $\Ro$.}
  \label{fig:P2K3} 
\end{figure}

The trivial cases $K_1$ and $K_2$ are explained by the trees $K_1$ and
$K_2$ with label $\lambda(e)=k$ at the unique edge $e$, respectively. For
$|L|=3$ there is only a single phylogenetic tree, the star $S_3$ with three
leaves and two connected graphs, $P_3$, and $K_3$, see
Fig.~\ref{fig:P2K3}. We denote the edges from the center to leaf $x_i$ by
$e_i$, $1\le i\le 3$. Fig.~\ref{fig:P2K3} also shows that class of graphs
explained w.r.t.\ $\Rt$ is much larger than the exact-2-leaf powers graphs,
which comprise only the disjoint unions of cliques
\cite{Nishimura:02,Brandstaedt:10}.

\begin{lemma}
  There are unique labelings $\lambda_{P_3}$ and $\lambda_{K_3}$ of the
  tree $S_3$ with discrete $\Ro$ that explain the graphs  
  $P_3$ and $K_3$, respectively:\\
  $\lambda_{P_3}(e_1)=\lambda_{P_3}(e_2)=2$ and $\lambda_{P_3}(e_3)=0$;\\
  $\lambda_{K_3}(e_1)=\lambda_{K_3}(e_2)=\lambda_{K_3}(e_3)=1$;\\
\end{lemma}
\begin{proof}
  $S_3$ contains three paths on length two. 
  Adopting the notation of Fig.~\ref{fig:P2K3} for both cases $P_3$ and
  $K_3$ we need $\lambda(e_1)+\lambda(e_2)=2$ and
  $\lambda(e_3)+\lambda(e_2)=2$. Therefore $\lambda(e_1)\in\{0,1,2\}$. 
  Explicitly enumerating the three cases yields:\\
  $\lambda(e_1)=0$ implies $\lambda(e_2)=2$ and thus $\lambda(e_3)=0$,
  in which case $x_1\Ro x_3$, contradicting the fact $\Ro$ is discrete.\\
  $\lambda(e_1)=1$ implies  $\lambda(e_2)=\lambda(e_3)=1$, and thus 
  $G(\Rt)=K_3$.\\
  $\lambda(e_1)=2$ implies $\lambda(e_2)=0$ and thus $\lambda(e_3)=2$,
  whence $G(\Rt)=P_3$.
\end{proof}

\begin{lemma} 
  \label{lem:P4} 
  The path $P_4$ on four vertices $x-y-z-u$ is explained only be the tree
  $T=(xy)p-q(zu)$ with labels $\lambda(xp)=\lambda(qu)=\lambda(pq)=2$ and
  $\lambda(yp)=\lambda(qz)=0$.
\end{lemma}
\begin{proof}
  First we observe that the path $P_4$ on four vertices cannot be explained
  by any labeling of a $S_4$. This leaves the fully resolved tree on four
  vertices. Its interior edge $pq$ cannot be labeled $0$. First consider
  $\lambda(pq)=1$. It cannot contain an $S_3$ with all three edges labeled
  $1$ since this would induce a triangle, i.e., at most one neighbor of
  $p$, say $x$, is attached by a 1-edge. The other neighbor of $p$, call it
  $y$, then must be attached by a $0$-edge, since otherwise $y$ is
  isolated. In order for $y$ not to be isolated, $q$ also must have a
  neighbor, that is connected via a $1$-edge, say $\lambda(qz)=1$. The same
  argument implies the the remaining leaf $u$ must be connected to $q$ with
  $\lambda(qu)=0$. This tree, however, explains the non-connected graph
  $K_2\cup K_2$. Thus $\lambda(pq)=2$. Connectedness implies that at least
  one of the leaves attached to $p$ and $q$ must be labeled $0$, say
  \NEW{$\lambda(py)=\lambda(qz)=0$}, and thus $y$ and $z$ are adjacent in
  $G$.  It remains to consider the possible coloring for the remaining to
  edges $\lambda(px)$ and $\lambda(qu)$. If $\lambda(px)=1$ then $x$ is
  isolated for all choices of $\lambda(qz)$. \NEW{An analogous statement is
    true for $\lambda(qu)=1$.}  For $\lambda(px)= \lambda(qu)=0$ we obtain
  $K_4-e$. If $\lambda(px)=0$ and $\lambda(qu)=2$ we obtain $S_3$. \NEW{The
    same is true for $\lambda(px)=2$ and $\lambda(qu)=0$.} Thus the only
  remaining choice is $\lambda(px)=\lambda(qu)=2$. It indeed explains the
  \NEW{path} $x-y-z-u$, \NEW{see Fig.~\ref{fig:4vertices}}.
\end{proof}

The fact that $S_n$ is the only ``exact-2-leaf root'' of $K_n$, i.e., the
only tree with unit edge weights that explains $K_n$ is shown in
\cite[Lemma 2]{Brandstaedt:10}. It is not difficult to see that there is
also no other choice of non-negative integer labels on $S_n$ that explains
$K_n$:
\begin{lemma}
  \label{lem:Sn}
  The complete graph $K_n$ is explained with respect to $\Rt$ by the star
  $S_n$ with the unique labeling function $\lambda(e)=1$ for all
  $e\in E(S_n)$.
  \label{lem:clique}
\end{lemma}
\begin{proof}
  Is is easy to check that this construction explains $K_n$ for all
  $n\ge 3$. The trivial cases $n=1$ and $n=2$ are explained in the
  text. \NEW{$K_3$ is only explained by $S_3$ with all edges labeled
    $\lambda(e)=1$. Since the start $S_n$ displays $S_3$ corresponding to
    every $K_3$ subgraph, all edges of $S_n$ must be labeled by
    $\lambda(e)=1$.}
\end{proof}

We note for later reference that the uniqueness results in Lemmas
\ref{lem:P4} and \ref{lem:clique} do not require the precondition that
$\Ro$ is discrete. This observation will be important in the following
section.

\begin{figure}
\begin{center}
  \includegraphics[width=0.9\textwidth]{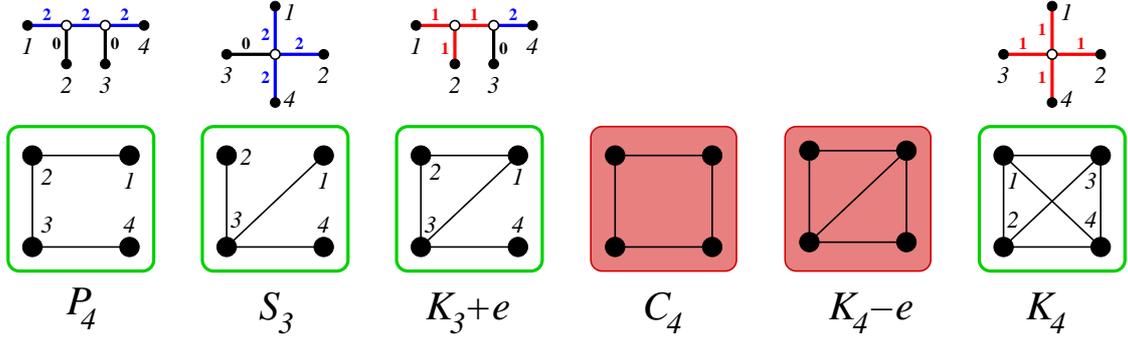}
\end{center}
\caption{Connected graphs on four vertices that are explained (green
  outline) or cannot be explained (red background) with respect to
  $\Rt$. The explaining trees are shown above the graphs.  The path $P_4$
  and the star $S_3$ are explained because all trees are already explained
  by $\Rl$. The $K_4$ is explained by Lemma~\ref{lem:clique}. The
  \NEW{graph} $K_3+e$ can be explained by a fully resolved phylogenetic
  tree with a edge labeling explicitly given in the proof of
  Lemma~\ref{lem:forbidden4}. In contrast, $C_4$ and $K_4-e$ cannot be
  explained with respected to $\Rt$ with an edge labeling with discrete
  $\Ro$ \NEW{according to Lemma~\ref{lem:forbidden4}}.}
  \label{fig:4vertices}
\end{figure}

\begin{lemma} 
  There is no edge-labeled tree $(T,\lambda)$ with discrete $\Ro$ that
  explains the graphs $C_4$ and $K_4-e$ with respect to $\Rt$.  The graph
  $K_3+e$ is explained by a unique edge-labeled tree.
  \label{lem:forbidden4}
\end{lemma}
\begin{proof} 
  There are two topologically distinct trees for $|L|=4$, the star $S_4$
  and tree $T_4$ with a single interior split. First consider the star
  $S_4$. In order to explain $K_3+e$ or $K_4-e$ three of the four edges
  must be labeled $1$ (corresponding to the induced $K_3$. Depending on
  whether $\lambda(e_4)=1$ or $\lambda(e_4)\ne 1$, the fourth vertex is
  either connected to all or none of the three other vertices. In order
  explain $C_4$, there must be two edges with $\lambda(e_1)=\lambda(e_3)=2$
  and one with $\lambda(e_2)=0$ corresponding to an induced $P_3$. The
  remaining edge then must have $\lambda(e_4)=0$. But then $x_2\Ro x_4$,
  contradicting that $\Ro$ is discrete.

  Now consider the tree$T_4$, which can be obtained from $S_3$ by subviding
  one of the edges and attaching an extra leaf to the subdividing
  vertex. Denote by $s$ the (unique) inner edge of $T_4$. Consider $K_3+e$
  and $K_4-e$ as shown in Fig.\ \ref{fig:4vertices}. Then we must have
  $\lambda(e_1)=\lambda(e_2)=1$. If $\lambda(s)=0$ we recover the situation
  of $S_4$, since the inner edge does not contribute to $\Rt$. On the other
  hand, if If $\lambda(s)=2$, then $x_1$ and $x_2$ cannot be connected with
  $x_3$ or $x_4$, contradicting the existence of $K_3$ as induced
  subgraph. Thus $\lambda(s)=1$. Then $\lambda(e_3)=0$. By assumption,
  $\lambda(e_4)\ne 0$ since otherwise $x_3\Ro x_4$. If $\lambda(e_4)=1$,
  the $x_4$ is an isolated vertex in $G$. If $\lambda(e_4)=2$, then
  $x_4\Rt x_3$ while $x_4$ is not in $\Rt$ relation to either $x_1$ or
  $x_2$. Thus $G=K_3+e$. The corresponding edge labeled tree is shown in
  Fig.~\ref{fig:P2K3}. Since we have already considered all cases, $K_4-e$
  cannot be explained with respect to $\Rt$.
 
  Finally, consider $T_4$ and suppose that $G$ contains $P_3$ as induced
  subgraph. There are two cases: If $\lambda(e_1)=\lambda(e_2)=2$ then
  connectedness of $G$ implies that
  $\lambda(s)=\lambda(e_3)=\lambda(e_4)=0$, contradicting that $\Ro$ is
  discrete. In the alternative case we can assume, w.l.o.g., that
  $\lambda(e_1)=2$ and $\lambda(e_2)=2$. Furthermore, in order to explain
  $C_4$ we must have $\lambda(e_3)+\lambda(e_4)=2$.  If both
  $\lambda(e_3)=\lambda(e_4)=1$. Then $\lambda(s)=0$ and $\lambda(s)=2$
  yields $G=K_2\cup K_2$, for $\lambda(s)=1$ we obtain $S_4$. In the
  remaining case we can choose $\lambda(e_3)=2$ and $\lambda(e_4)=0$. Now
  $\lambda(s)=0$ contradicts discreteness of $\Ro$, $\lambda(s)=1$ yields
  the edgeless graph. For $\lambda(s)=2$ we obtain $P_4$. Thus $C$ cannot
  be explained by $T_4$ with respect to $\Rt$ by a labeling with discrete
  $\Ro$.
\end{proof}

\begin{figure}[t]
  \centering
      \includegraphics[width=0.75\textwidth]{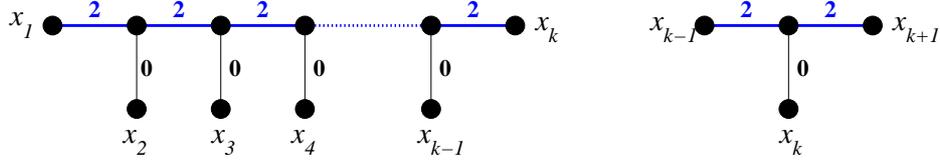}
  \caption{Construction of trees that explain paths.}
   \label{fig:path} 
\end{figure}

The fact that $K_4-e$ is a forbidden subgraph implies that two cliques in
$G$ cannot be ``glued together'' by a single common edges. It is possible,
however, for cliques to touch in a cut vertex as shown by the example of
the bowtie graph $B$, which is obtained by gluing together two triangles at
a common vertex, see Fig.~\ref{fig:P2K3}.

Graphs that can be represented as pairwise compatibility graphs of
caterpillars have received special attention in the literature
\NEW{\cite{PCGsurvey,Brandstaedt:08,Calameroni:14,Salma:13}}.  It is not
difficult to see that the path $P_h$, $h\ge 3$ can represented by a
caterpillar in several settings. These results cannot be directly applied
in \NEW{our} setting, however. Any two leaves $x$ and $y$ attached to two
distinct inner vertices of a caterpillar are separated by at least three
edges an thus cannot be in relation $\Rt$ if we assume strictly positive
integer weights. It follows immediately that $P_h$ is not an exact-2-leaf
power of a caterpillar and that $P_h$ cannot be explained by caterpillar
unless zero-weights are allowed. An explicit construction in
\cite{Hellmuth:17q} shows that $P_h$ is explained by a caterpillar with
edge weights in $\{0,1\}$ with respect to exactly-1-relation $\Rl$.  Lemma
\ref{lem:multiple} implies that we can use the same construction to explain
$P_h$ by a caterpillar with edge weights in $\{0,2\}$, see
Fig.~\ref{fig:path}. It will be important later on that this construction
is indeed unique:

\begin{lemma}
  \label{lem:Ph}
  The path $P_h$ has as its unique explaining tree the caterpillar
  $(T_h,\lambda_h)$ with all inner edges and the edges connecting to the
  end-points of $P_h$ labeled $2$ and all edges connecting to inner
  vertices of $P_h$ labeled $0$.
\end{lemma}
\begin{proof}
  We first recall that the tree $(T_4,\lambda_4)$ explaining $P_4$ is
  unique by Lemma~\ref{lem:P4}.  Now assume that for $h\ge 5$, the tree
  $(T_{h-1},\lambda_{h-1})$ explaining $P_{h-1}$ is unique and thus a
  caterpillar. Any tree $(T,\lambda)$ explaining $P_h$ therefore must
  display $(T_{h-1},\lambda_{h-1})$, i.e., $(T,\lambda)$ is obtained from
  $T_h$ by subdiving one edge and attaching leaf $h$ and edge $e_h$ to the
  new vertex, or by attaching $h$ and $e_h$ to an inner vertex of
  $T_{h-1}$. One easily checks that the latter yields a branched tree or a
  disconnected graph. The same is true is any other edge except $e_{h-1}$
  and $e_1$, the edges adjacent to the leaves $h-1$ or $1$ are
  subdivided. In the latter case, $h$ cannot be adjacent to $h-1$. In the
  remaining case, the edge with which $h-1$ is attached is subdivided into
  an interior part $s$ and the part $e_{h-1}$ incident with $h-1$. Since
  the interior part cannot carry a zero label, we must have
  $\lambda_h(e_{h-1})=0$, $\lambda_h(s)=2$, and $\lambda_h(e_h)=2$. Thus
  the caterpillar of Fig.~\ref{fig:path} is indeed the only choice.  We
  emphasize that this observation remains true even is $\Ro$ is not assumed
  to be discrete.
\end{proof}

\begin{lemma}
  \label{lem:C_p}
  The simple cycles $C_p$, $p\ge 5$ cannot be explained with respect to
  $\Rt$ irrespective of whether $\Ro$ is discrete or not.
\end{lemma}
\begin{proof}
  Every cycle $C_p$ contains a path $P_{p-1}$ with one vertex less as an
  induced subgraph. From Lemma~\ref{lem:Ph} we known that $P_{p-1}$ has a
  unique explanation by a caterpillar for all $p\ge 5$. Thus any tree
  $(T^*,\lambda^*)$ explaining $C_p$ thus must display caterpillar
  $(T_{p-1},\lambda_{p-1})$ and thus $T^*$ is obtained from $T_{p-1}$ by
  either attaching $p$ and $e_p$ to inner vertex of $T_{p-1}$ or by
  subdiving an edge and attaching $p$ and $e_p$ to the newly inserted
  vertex. As argued above, attachment to an inner vertex or subdivision of
  an edge other than $e_1$ or $e_{p-1}$ leads to a branched tree or a
  disconnected graph. If $e_p$ is inserted by subdivision of $e_1$, then
  $p$ cannot be adjacent to $p-1$ and subdivision of $e_{p-1}$ precludes
  adjacency of $p$ and $1$ for $p\ge 3$. Thus the catapillar tree
  $(T_{p-1},\lambda_{p-1})$ cannot be extented to tree that explains $C_p$
  for any $p\ge 4$. Note that this argument did not make the assumption
  that $\Ro$ is discrete.
\end{proof}

Let us now turn to the general case.  We first note that all graphs
\NEW{explained w.r.t.\ $\Rt$ with discrete $\Ro$} are chordal, i.e., every
cycle of length greater than three has a chord. Even more stringently,
every cycle of length $4$ corresponds to a clique in $G$ because the
$K_4-e$, i.e., the 4-cycle with a chord, is also a forbidden induced
subgraph. We note that there is ample literature on the relationship of
chordal graphs and PCGs, see e.g.\ \cite{Yanhaona:09,PCGsurvey}. Due to the
differences in the edge weight functions, it is not immediately pertinent
to our discussion, however.

\begin{lemma}
\label{lem:Hamilton} 
If $G$ can be explained by the exact-$2$-relation with discrete $\Ro$ and
contains a Hamiltonian cycle, then $G$ is a complete graph.
\end{lemma}
\begin{proof}
  The assertion is trivially true for $n=3$ and holds for $n=4$ because 
  $C_4$ and $K_4-e$, the only Hamiltonian graphs on 4 vertices except $K_4$ 
  are forbidden induced subgraphs. Now suppose the statement is true for 
  for all $|V|<p$ and consider a graph with $p$ vertices. Since $G$ is 
  chordal, there is in particular a planar triangulation of $C$ that is 
  a subgraph of $G$ and thus there are three consecutive vertices 
  $u-v-w$ along $C$ such that $u-w$ is a also an edge in $G$. Thus 
  $G\setminus v$ is Hamiltonian. As an induced subgraph of $G$ it can be 
  explained by the exact-$2$-relation and thus is a complete graph by the 
  induction hypothesis. Thus $u-x-w$ is triangle in $G\setminus v$ and 
  $u-x-w-v$ is a cycle of length $4$ in $G$. Since $C_4$ and $K_4-e$ cannot 
  appear as induced subgraphs of $G$, $\{u,v,w,x\}$ must for a clique in
  $G$, and hence the edge $\{v,x\}\in E(G)$ for all $x\in V(G\setminus v)$.
  Thus $G$ is a complete graph. 
\end{proof}

\begin{lemma} 
  A graph $G$ with at least three vertices that can be explained by the
  exact-2-relation with discrete $\Ro$ is complete if and only if it is
  2-connected.
  \label{lem:2conn-clique}
\end{lemma}
\begin{proof}
  If $G$ is Hamiltonian, it is in particular also 2-connected. Now consider
  the case that $G$ is 2-connected but not Hamiltonian. Let $C$ be a cycle
  of maximal length in $G$ and let $x$ be a vertex not in $C$. Then there
  is a cycle $C'$ in $G$ that contains $x$ and at least two distinct
  vertices of $C$ since otherwise one of the vertices of $C$ would be a cut
  vertex of $G$, contradicting 2-connectedness. Starting from $x$, let $p$
  and $q$ be first and last vertex of $C$ encountered along $C'$. By
  Lemma~\ref{lem:Hamilton}, $G[C]$ is a complete graph, and hence there is
  a another Hamiltonian cycle $C''$ on $G[C]$ so that $p$ and $q$ are
  consecutive along $C'$. Thus the cycle $C^*$ obtained traversing $C''$
  from $p$ to $q$ and then following $C'$ from $q$ through $x$ back to $p$
  is a cycle that is strictly longer than $C$, contradicting maximality.
  Thus $G$ is Hamiltonian, and hence complete. 
\end{proof}

A graph $G$ is a \emph{block graph} \cite{Harary:63} if each of its
biconnected components is a clique. Lemma~\ref{lem:2conn-clique} thus
implies that every graph that can be explained with respect to the
exact-$2$-relations with discrete $\Ro$ is a block graph \NEW{(see Thm.\
  \ref{thm:block} below for a formal proof)}.
Algorithm~\ref{alg:blockgraph} (illustrated in Figure~\ref{fig:blockgraph})
\NEW{explicitly constructs an edge-labeled tree that explains a given block
  graph.}

\begin{algorithm}
\caption{Compute $(T(G),\lambda)$ for a connected block graph $G$}
\label{alg:blockgraph}
\begin{algorithmic}[1]
  \REQUIRE a connected block graph $G$ 
  \STATE mark ``red'' all cut vertices $u\in V(G)$.
  \FORALL {cliques $K$ in $G$} 
     \IF { $K$ is an edge $e$ } 
        \STATE $\lambda(e)=2$ 
     \ELSE 
        \STATE replace $K$ by a star $S_{|V(K)|}$ with center $c_K$
        \STATE $\lambda(u c_K)=1$ for each $u\in V(K)$
     \ENDIF
  \ENDFOR
  \FORALL {red vertices $v$}
     \STATE add a vertex $v'$ and edge $vv'$ with $\lambda(vv')=0$
     \STATE exchange the vertex names $v$ and $v'$
  \ENDFOR
  \RETURN $(G,\lambda)$
\end{algorithmic}
\end{algorithm}

\begin{lemma} 
\label{lem:algorithm}
Algorithm~\ref{alg:blockgraph} \NEW{transforms any connected block-graph
  $G$ into} an edge-labeled tree \NEW{that explains} $G$ with respect to
the exactly-$2$-relation \NEW{$\Rt$} with discrete $\Ro$.
  \label{lem:algo} 
\end{lemma}
\begin{proof}
  The output of Alg.~\ref{alg:blockgraph} contains no cycles since all
  cycles in the input $G$ are contained within a block and each block is
  replaced by a star. Furthermore, the replacement of a clique $K_p$ by a
  star $S_p$ with $p+1$ vertices preserves connectedness, hence $G$ has
  been transformed into a tree at this stage. Every vertex of a clique $K$,
  with $|V(K)|\ge 3$ this is not also contained in another block is now a
  leaf; all other nodes of $K$ are marked red. Every vertex in an $K_2$
  original $K_2$ block is either a leaf or marked ``red''. By construction,
  every ``red'' vertex has degree at least $2$ and hence is not a leaf.
  The final operation adds a leaf to each ``red'' vertex. Together with the 
  renaming of the vertices, thus, every vertex of the input graph is now a
  leaf in $T$. 

  Now consider the labeling. First suppose that $u$ and $v$ are
  non-adjacent in the input $G$, that is, there is a least one cut-vertex,
  say $z$, between them in $G$. The construction of $T(G)$ ensures that the
  unique path from $u$ to $v$ in $T(G)$ runs through a vertex $z'$ that $z$
  as its neighbor. If the path from $u$ to $z$ in $G$ ran through an edge
  in a triangle, it passes through the corresponding star and hence
  contains two edges labeled $1$. Otherwise it runs through an unaltered
  $K_2$-block of $G$, which is labeled $2$. In each case, therefore,
  $d_{T(G),\lambda}(x,y)\ge 4$. Now suppose that $u$ and $v$ are adjacent
  in $G$. First suppose $uv$ is contained in a triangle of $G$.  If neither
  $u$ nor $v$ was marked ``red'' they are both adjacent to the center $c_K$
  of a star with edges labeled $1$. If $u$ was a cut vertex, i.e., marked
  ``red'', it appears a leaf adjacent to a vertex $u'$ that in turn is
  adjacent to $c_K$; furthermore $\lambda(uu')=0$ and
  $\lambda(u'c_K)=1$. Analogous reasoning applied if $v$ was a cut vertex
  of $G$. In all cases, thus
  $d_{T(G),\lambda}(uv)=d_{T(G),\lambda}(u,c_K)+d_{T(G),\lambda}(c_K,v)=1+1=2$.
  If the edge $uv$ is not contained in a triangle, then it is labeled $2$.
  If $u$ or $v$ are cut vertices, then the unique path from $u$ to $v$ is
  $u-u'-v$, $u-v'-v$, or $u-u'-v'-v$, with $\lambda(uu')=\lambda(vv')=0$ and
  a label $2$ for the remaining edge. Hence, $d_{T(G),\lambda}(uv)=2$.  In
  summary $u\Rt_{T(G),\lambda} v$ if and only $u$ and $v$ are adjacent in
  $G$. Thus $G$ is explained by $(T(G),\lambda)$ with respect to the
  exactly-$2$-relation.
\end{proof} 

\begin{figure}
\begin{center}
  \includegraphics[width=0.8\textwidth]{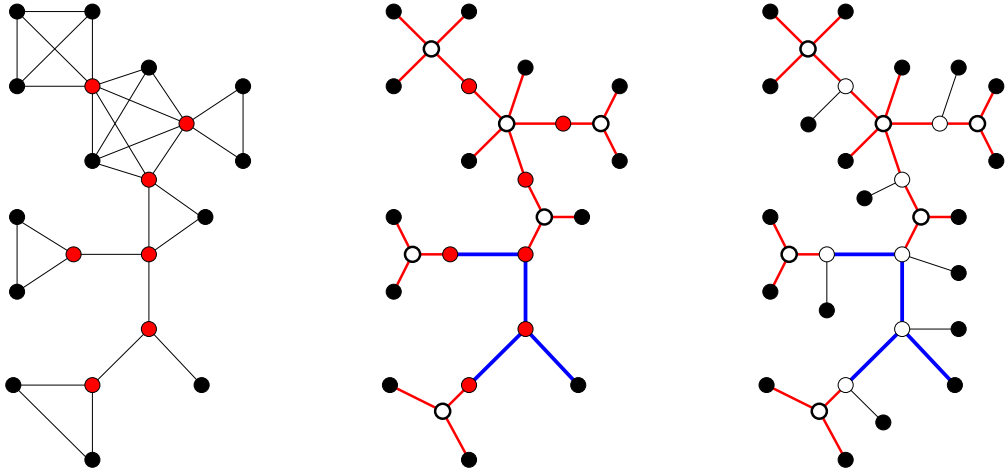} 
\end{center}
\caption{Illustration of Alg.~\ref{alg:blockgraph}. First all cut vertices
  of the input graph are marked. In the next step, edges not contained in
  triangles of $G$ receive label $2$ and all larger cliques are replaced by
  stars with all new edges labeled $1$. In the final step the original cut
  vertices are decorated by an additional neighbor with a $0$-labeled edge.}
\label{fig:blockgraph}
\end{figure}

\begin{thm} \label{thm:block} A graph $G$ can be explained by an
  edge-labeled tree $(T,\lambda)$ with respect to the exact-$2$-relation
  with discrete $\Ro$ if and only if it is \NEW{a} block graph.
\end{thm}
\begin{proof}
  \NEW{Suppose $G$ can be explained w.r.t.\ $Rt$ with discrete $\Ro$.  If
    $G$ is 2-connected, it is a clique by Lemma~\ref{lem:2conn-clique} and
    therefore also a block graph. Otherwise, we note that every 2-connected
    component $G'$ of $G$ is induced subgraph of $G$ and thus, by
    Lemma~\ref{lem:hered}, can be explained w.r.t.\ $Rt$. By
    Lemma~\ref{lem:2conn-clique} every 2-connected component $G'$ of $G$
    therefore must be a clique, i.e., $G$ is a block graph.}
    
  \NEW{Conversely, suppose that $G$ is a block graph.}  Since
  Algorithm~\ref{alg:blockgraph} is correct by Lemma~\ref{lem:algo}, every
  connected block graph can be explained.  Since the non-connected block
  graphs are just disjoint unions of connected block graphs,
  Lemma~\ref{lem:connected} completes the characterization of the
  non-connected case.
\end{proof} 

The main result of this section is now obtained as
\begin{cor} 
  A graph $G$ is explained by $\Rt$ if and only if $R/\!\rthin$ is a block
  graph. 
\end{cor} 
\begin{proof} 
  Thm.~\ref{thm:thin} establishes that $G$ can be explained w.r.t.\ $\Rt$
  if and only if $R/\!\rthin$ can be explained w.r.t.\ $\Rt$. Since the graph
  $G/\!\rthin$ is thin, Cor.~\ref{cor:Ro-thin} implies that $\Ro$ is discrete
  for the canonical tree explaining $G/\!\rthin$, and thus
  Thm~\ref{thm:block} can be applied to $G/\!\rthin$.
\end{proof}

In the remainder of this section we consider the ambiguities in the
construction of trees explaining block graphs. We start \NEW{by}
characterizing contractible edges:
\begin{lemma}
  \label{lem:least}
  Suppose $(T_e,\lambda_e)$ is obtained from a phylogenetic tree
  $(T,\lambda)$ by contracting the edge $e$ in $T$ and setting
  $\lambda_e(e')=\lambda(e')$ for all $e'\ne e$ and suppose that
  $G(T,\lambda)$ is connected. Then $G(T,\lambda)=G(T_e,\lambda_e)$ if and
  only if $e$ is an interior edge of $T$ and $\lambda(e)=0$.
\end{lemma} 
\begin{proof}
  We have already noted the contracting an inner $0$-edge does not change
  the graph. By definition, leaf-edges cannot be contracted, since the
  vertices of $G$ correspond to the leaves of $T$. Connectedness of $G$
  implies that there is a pair of vertices $x,y$ whose connecting path runs
  through $e$ and whose distance \NEW{$d_{T,\lambda}(x,y)=2$}.  The
  contraction of $e$ only leaves this distance unaffected if
  $\lambda(e)=0$. Otherwise \NEW{$d_{T,\lambda}(x,y)$} changes, which
  implies that $x$ and $y$ become disconnected in $G'$ and hence the graph
  by the modified tree is different from $G$.
\end{proof}

\begin{figure}[b]
  \begin{center}
    \includegraphics[width=0.8\textwidth]{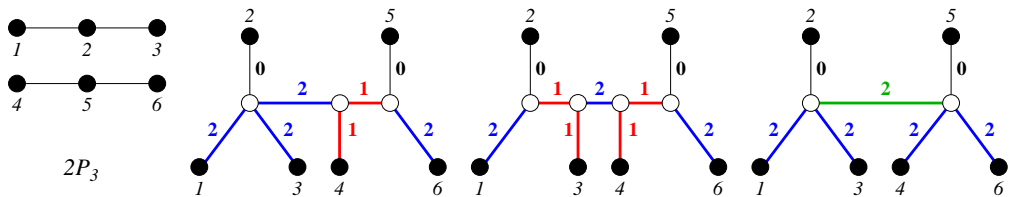}
  \end{center}
  \caption{Three alternative, topologically different canonical trees
    explaining the non-connected graph $2P_3$. These are derived from the
    unique tree explaining $P_3$ in Fig.~\ref{fig:P2K3}.}
  \label{fig:2P3}
\end{figure}

We note that connectedness of $G$ is necessary in Lemma~\ref{lem:least}
since for non-connected $G$, the connected components can be ``glued
together'' with arbitrarily complex trees as long as the distances between
the attachment points is at least $3$. In such examples it can be possible
to contract edges without changing the explained graph. There are, for
example at least three topologically different canonical trees that explain
$2P_3$, see Fig.~\ref{fig:2P3}.

\begin{lemma}
  Let $(T,\lambda)$ be a canonical tree explaining a connected graph $G$
  and let $x$ be an interior vertex in $T$. Then all edges incident to $x$
  are 1-edges or $p$ has at least one adjacent leaf $u$ with
  $\lambda(pu)=0$.
\end{lemma} 
\begin{proof}
  Suppose $p$ has no incident leaf. Since $G$ is connected, for every edge
  $e'$ there is another edge $e''$ such that
  $\lambda(e')+\lambda(e'')=2$. For this pair of edges we have
  $\lambda(e')=\lambda(e'')=1$ because no interior edges is
  $0$-labeled. Thus $\lambda(e')=1$ for all $e'$ incident with $x$.  On the
  other hand, if $p$ has a neighbor $u$ with $\lambda(pu)=2$, then
  connectedness of $G$ implies that there is another neighbor $y$ of $p$
  with $\lambda(yp)=0$. Hence, unless $p$ has only $1$-neighbors, then
  there must be a least one incident $0$-edge, which by assumption must be
  a leaf.
\end{proof}

We remark, finally, that a tree with minimal number of vertices (or edges)
that explains a graph with respect to the exactly-$2$-relation is
necessarily canonical.  Otherwise, the contraction of an edge would make it
possible to decrease of both the number of edges and vertices.
  
\section{Oriented Exactly-2-Relation}

Generalizing the construction of the oriented exactly-1-relation in
\cite{Hellmuth:17q}, we consider here an oriented version of the
exactly-$k$-relation. In constrast to the previous sections, we consider
here rooted trees $T$ with leaf set $L$. For two leaves $x$ and \NEW{$y$}
there is a unique \emph{least common ancestor}, denoted by $\lca{x,y}$,
defined as the vertex most distant from the root $r$ of
$\overrightarrow{T}$ that is common to the paths connecting $r$ with $x$
and $r$ with $y$, respectively.

\begin{defi} 
  Let $(\overrightarrow{T},\lambda)$ be a rooted tree with leaf set $L$ and
  edge-labeling function $\lambda: E(\overrightarrow{T})\to\mathbb{N}_0$.
  For $x,y\in L$ we consider the \emph{directed exactly-$k$-relation
    $\Rkd$} defined by \emph{$x \Rkd y$} if
  $\sum_{e\in\mathcal{P}(x,\lca{x,y})} \lambda(e)=0$ and
  $\sum_{e\in\mathcal{P}(\lca{x,y},y)} \lambda(e)=k$ holds for the the
  (unique) paths $\mathcal{P}(x,\lca{x,y})$ from $x$ to $\lca{x,y}$ and
  $\mathbb(\lca{x,y},y)$ from $\lca{x,y}$ to $y$, respectively.  
  \newline
  The rooted tree $(\overrightarrow{T},\lambda)$ \emph{explains} a the
  directed graph $\overrightarrow{G}(L,E)$ \emph{(with respect to the
    directed exactly-$k$-relation)} if $(x,y)\in E(G)$ if and only if
  $x\Rkd y$.
\end{defi}
By construction $\overrightarrow{G}(L,E)$ is an \emph{oriented} graph,
i.e., at most one of $(x,y)$ and $(y,x)$ can be an edge. As in the unrooted
case, we say that a rooted tree $(\overrightarrow{T},\lambda)$ is
\emph{canonical} if it is a rooted phylogenetic tree and does not have an
inner 0-edge.  In the following we will consider the case that $\Ro$ is
discrete. As in the undirected case, we shall relax this requirement in the
end.

As in \cite{Hellmuth:17q}, our strategy is to exploit the close
relationships between the oriented and the undirected case. Therefore, we
first derive some technical results regarding common properties of the
oriented relation $\Rkd$ and its undirected relative $\Rk$.

Note that the underlying tree $(T,\lambda)$ of a rooted canonical tree
$(\overrightarrow{T},\lambda)$ is not necessarily an unrooted canonical
tree. By contracting all the interior 0-edges and degree 2 vertices, we get
a unique unrooted canonical tree $(T',\lambda')$ corresponds to
$(\overrightarrow{T},\lambda)$. Conversely, for any unrooted canonical tree
$(T,\lambda)$ with $|V(T)|>1$, we can create a set $\mathbb{T}(T,\lambda)$
of corresponding rooted least resolved trees as follows: (i) each interior
vertex of $(T,\lambda)$ may serve as a root; (ii) each leaf attached by a
0-edge may serve as a root; and (iii) every 2-edge can be subdivided by
inserting a the root as a new vertex such that each of the two resulting
edges is labeled 1.  The construction is detailed in
Algorithm~\ref{alg:rooted}. An example is given in Figure~\ref{fig:roots}.
The following lemma formalizes this one-to-one correspondence between
unrooted canonical trees $(T,\lambda)$ and its corresponding sets of rooted
canonical tree.

\begin{algorithm}
  \caption{Compute the set of canonical rooted trees $(\mathbb{T,\lambda})$}
  \label{alg:rooted}
  \begin{algorithmic}[1]
  \REQUIRE unrooted canonical tree $(T,\lambda)$ with $|V(T)|>1$
  \STATE $(\mathbb{T,\lambda}) \leftarrow \emptyset$
  \FORALL {interior vertices $v\in T$} 
     \STATE designate $v$ as root 
     \STATE add the rooted tree to $(\mathbb{T,\lambda})$
     \ENDFOR
  \FORALL {leaf vertices $v\in T$ with $\lambda(vw)>0$ where $N(v)=\{w\}$}
     \STATE subdivide $vw$ to $vv^*w$ and designate $v^*$ as root
     \STATE relabel as $\lambda(vv^*)\leftarrow\lambda(vw)$ and 
            $\lambda(v^*w)\leftarrow 0$, designate $v^*$ 
     \STATE add the resulting rooted tree to $(\mathbb{T,\lambda})$.
  \ENDFOR   
  \FORALL {edges $e=uv$ with $\lambda(e)=k>1$}
     \STATE subdivide the edge $e$ by inserting $v^*$ and designate $v^*$ as
            the root
     \FOR {$j=1...k-1$} 
        \STATE $\lambda(uv^*)\leftarrow j$ and 
            $\lambda(v^*v)\leftarrow k-j$
        \STATE add the resulting rooted tree to $(\mathbb{T,\lambda})$.
     \ENDFOR 
  \ENDFOR 
  \RETURN $(\mathbb{T,\lambda})$
\end{algorithmic}
\end{algorithm}

\begin{lemma} \label{lem:allrooted} 
  Every rooted canonical tree can be constructed from its underlying
  unrooted canonical tree by Algorithm ~\ref{alg:rooted}.
\end{lemma}
\begin{proof}
  By construction, the set of canonical rooted trees corresponding to
  unrooted canonical tree is well defined, i.e., the correspondence is a
  mapping.

  Suppose there are two distinct unrooted canonical trees $(T_1,\lambda_1)$
  and $(T_2,\lambda_2)$ such that both their correspondings sets of rooted
  trees contains $(\overrightarrow{T},\lambda)$. By construction, it has a
  underlying tree $(T,\lambda)$ from which a unique canonical tree is
  obtained by contracting 0-edges and degree 2 vertices. Thus
  $(T_1,\lambda_1)=(T_2,\lambda_2)$, a contradiction. Hence the mapping is
  injective.
  
  The mapping is also surjective, since each rooted canonical tree
  $(\overrightarrow{T},\lambda)$ can be constructed from its corresponding
  unrooted canonical tree.
\end{proof}

\begin{lemma} \label{lem:subgraph} Suppose the unrooted canonical tree
  $(T,\lambda)$ explains $G$ with respect to $\Rk$. Let
  $\overrightarrow{G}$ be a digraph explained w.r.t.\ $\Rkd$ by a rooted
  tree $(\overrightarrow{T},\lambda)$ corresponding to $(T,\lambda)$.
  Then the underlying graph of $\overrightarrow{G}$ is a spanning 
  subgraph of $G$.
\end{lemma}

\begin{proof}
  By construction, $(T,\lambda)$ and $(\overrightarrow{T},\lambda)$ has the
  same leaf set, and hence $V_G= V_{\overrightarrow{G}}$.
 
  Any arc $x\to y$ in $\overrightarrow{G}$ is an edge in the underlying
  graph of $\overrightarrow{G}$ because that fact that
  $(\overrightarrow{T},\lambda)$ explains $\overrightarrow{G}$ implies
  $\sum_{e\in\mathcal{P}(x,\lca{x,y})} \lambda(e)=0$ and
  $\sum_{e\in\mathcal{P}(\lca{x,y},y)} \lambda(e)=k$. Considering the
  underlying unrooted graph $(T',\lambda')$, we have
  $\sum_{e\in\mathcal{P}(x,y)} \lambda'(e)=k$. Since $(T,\lambda)$ explains
  $G$, and by construction $(T',\lambda')$ displays $(T,\lambda)$, we
  conclude that $(T',\lambda')$ also explains $G$. Hence $(x,y)\in E(G)$.
\end{proof}

\begin{figure}[htbp!]
  \begin{center}
  \includegraphics[width=\textwidth]{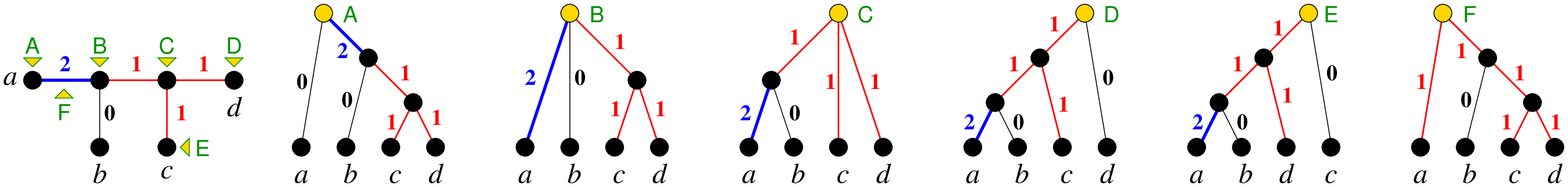}
  \end{center}
  \caption{Construction of rooted canonical trees \textsf{A} through
    \textsf{F} from the unrooted canonical tree on the left. The possible
    positions of the root are indicated by the triangles and the
    corresponding six rooted trees are shown to the right.}
  \label{fig:roots}
\end{figure}

\begin{defi}
  Suppose there a tree $(T,\lambda)$ with discrete $\Ro$ that explains $G$
  w.r.t.\ $\Rk$, then every subgraph $H$ of $G$ is \emph{allowed} for
  $\Rk/\Ro$. Analogously, if there is a rooted tree
  $(\overrightarrow{T},\lambda)$ with discrete $\Ro$ that explains
  $\overrightarrow{G}$ w.r.t.\ $\Rkd$, we say that every subgraph
  $\overrightarrow{H}$ of $\overrightarrow{G}$ is allowed for $\Rk/\Ro$.
\end{defi}
In more detail a graph $H$ is allowed for $\Rk/\Ro$ if there exists
$(T,\lambda)$ such that for any $(x,y)\in E(G)$, we have
$\sum_{e\in\mathcal{P}(x,y)} \lambda(e)=k$.  If $G$ is not allowed for
$\Rk/\Ro$, we say that is it is \emph{forbidden (as a subgraph)} for
$\Rk/\Ro$. Analogous, a graph $\overrightarrow{H}$ is allowed for
$\Rkd/\Ro$ in the rooted case, if there exists $(T,\lambda)$ such that for
any $(x,y)\in E(G)$, we have $\sum_{e\in\mathcal{P}(x,\lca{x,y})}
\lambda(e)=0$ and $\sum_{e\in\mathcal{P}(\lca{x,y},y)} \lambda(e)=k$.  If
$\overrightarrow{G}$ is not allowed as a subgraph in $G(\Rkd)/\Ro$, we that
say $G$ is \emph{forbidden (as a subgraph)} for in $\Rkd/\Ro$.

\begin{lemma}
  \label{lem:underly}
  If $G$ is forbidden for $\Rk/\Ro$, then any orientation of $G$ is
  forbidden in $\Rkd/\Ro$.  If $\overrightarrow{G}$ is allowed as a
  subgraph for $\Rkd/\Ro$ with rooted tree $(\overrightarrow{T},\lambda)$,
  then its underlying graph is allowed for $\Rk/\Ro$ as a subgraph with the
  corresponding underlying tree $(T,\lambda)$.
\end{lemma}

\begin{proof}
  Suppose, for contradictions, that $G$ is forbidden for $\Rk/\Ro$ but the
  orientation $\overrightarrow{G}$ of $G$ is allowed for $\Rkd/\Ro$. Then
  there exists a rooted tree ${\overrightarrow{T},\lambda}$ such that for
  any arc $x\to y$ in $\overrightarrow{G}$, we have
  $\sum_{e\in\mathcal{P}(x,u)} \lambda(e)=0$ and
  $\sum_{e\in\mathcal{P}(u,y)} \lambda(e)=k$ where $u=\lca{x,y}$.  Consider
  the unrooted tree $(T,\lambda)$ of $(\overrightarrow{T},\lambda)$. Since
  $(x,y)\in E(G)$ if and only if $x\to y$ or $y\to x$ is an arc in
  $\overrightarrow{G}$, then for any $(x,y)\in E(G)$,
  $\sum_{e\in\mathcal{P}(x,u)} \lambda(e)=0$ and
  $\sum_{e\in\mathcal{P}(u,y)} \lambda(e)=k$ where $u=\lca{x,y}$. By
  definition, $G$ is allowed for $\Rk/\Ro$, i.e., we arrive at a
  contradiction.  The second statement is a simple consequence of the first
  one.
\end{proof}

The technical results obtained so far will allow us to infer properties of
the oriented graph $\overrightarrow{G}$ and their explaining trees
$(\overrightarrow{T},\lambda)$ from their underlying undirected graphs $G$
and unrooted trees $(T,\lambda)$. In the following we will focus on graphs
$\overrightarrow{G}$ that can be explained w.r.t.\ $\Rtd$ by a rooted tree
$(\overrightarrow{T},\lambda)$ with discrete $\Ro$.

\begin{lemma} 
  \label{lem:cycle}
  Oriented cycles are forbidden as a subgraph for $\Rtd/\Ro$.
\end{lemma}
\begin{proof}
  Suppose $\overrightarrow{C_n}$ is allowed. Then, by definition, there
  exists an orientation graph $\overrightarrow{H}$ with vertex set
  $V(\overrightarrow{C_n})$ such that $\overrightarrow{C_n}$ is a subgraph
  of $\overrightarrow{H}$ and a rooted tree $(\overrightarrow{T},\lambda)$
  that explains $\overrightarrow{H}$.  W.l.o.g., we assume
  $(\overrightarrow{T_n},\lambda)$ is a rooted canonical tree.
 
  Consider the underlying unrooted canonical tree $(T,\lambda)$ of
  $(\overrightarrow{T_n},\lambda)$, we claim that $(T,\lambda)$ must be
  $(S_n,1)$.  Suppose that $(T,\lambda)$ explains graph $G$. By
  Lemma~\ref{lem:subgraph} since the underlying graph $H$ of
  $\overrightarrow{H}$ is a subgraph of $G$ which has the same vertex set
  with $H$, and $H$ contains a Hamiltonian cycle, thus $G$ also contains a
  Hamiltonian cycle, by Lemma \ref{lem:Hamilton} $G$ is a complete graph
  $K_n$.  And by Lemma~\ref{lem:Sn} the $(T,\lambda)$ displays $(S_n,1)$.
 
  Then we consider all the possibility to construct the set of rooted
  canonical trees corresponding to $(S_n,1)$, and consider the oriented
  graph it explains.  By Algorithm~\ref{alg:rooted} we can place the root
  either on the center vertex, which will explain an empty graph, or place
  it on one of the leaves, which will explains oriented star on $n$
  vertices point to the leaves. In either case there is no cycles.  By
  Lemma~\ref{lem:allrooted} we know we have constructed all rooted
  canonical trees and thus all oriented graphs they explain. Thus oriented
  cycles are forbidden as a subgraph for $\Rtd/\Ro$.
\end{proof}

\begin{lemma} \label{lem:forb} 2-star oriented to center,
  $\bullet\to\bullet\leftarrow\bullet$, is forbidden as an induced
  subgraph for $\Rtd/\Ro$.
\end{lemma}
\begin{proof}
  Explicit construction shows that we obtain $v_1\Ro v_3$ for each of the
  three triples $v_1v_2|v_3$, $v_1v_3|v_2$, and $v_2v_3|v_1$. This
  contradicts the assumption that $\Ro$ is discrete.
\end{proof}

\begin{lemma}
  \NEW{Every graph $\overrightarrow{G}$ that can be explained by an
    edge-labeled tree w.r.t.\ $\Rtd/\Ro$ is} an oriented forest \NEW{with
    the property that all its} component trees have a unique source vertex
  from which all arcs are directed away.
\end{lemma} 
\begin{proof}
  \NEW{Let $\overrightarrow{G}$ be a graph that can be explained w.r.t.\
    $\Rtd/\Ro$.}  Since all cycles are forbidden induced subgraphs,
  $\overrightarrow{G}$ is a forest. Furthermore, there is only a single
  source vertex in each connected component. Otherwise, if both $x$ and $y$
  were sources within the same component tree, then the unique path from
  $x$ to $y$ would necessarily contain an induced subgraph of the form
  $\bullet\to\bullet\leftarrow\bullet$, which is forbidden.
\end{proof}

A canonical tree $(\overrightarrow{T},\lambda)$ with discrete $\Ro$ that
explains a connected oriented graph $\overrightarrow{G}$ w.r.t.\ $\Rtd$ has
a leaf that is attached to the root by 0-edge.

\begin{thm}
  $\overrightarrow{G}$ is explained w.r.t.\ $\Rtd$ by a rooted tree
  $(\overrightarrow{T},\lambda)$ with discrete $\Ro$ if and only if $G$ is
  an oriented forest that does not contain the 2-star oriented to center,
  $\bullet\to\bullet\leftarrow\bullet$, as an induced subgraph.
\end{thm}
\begin{proof}
  To show the ``only if'' part, suppose $\overrightarrow{G}$ can be
  explained. Then $\overrightarrow{G}$ is oriented and by
  Lemma~\ref{lem:cycle} it is an orientation tree, and by
  Lemma~\ref{lem:forb} we know that 2-star oriented to center is forbidden.
  For the ``if'' part we use the construction employed in
  \cite{Hellmuth:17q} for $\Rl$ (with 2-edges taking the place of 1-edges):
  To each inner vertex $v$ of $\overrightarrow{G}$ a new vertex $v'$ which
  represent $v$ in tree is attached with a 0-edge, while the inner edges of
  the tree have label $2$.  The Theorem now follows directly from
  Lemma~\ref{lem:1tok}.
\end{proof}

We can relax the condition that $(\overrightarrow{T},\lambda)$ has discrete
$\Ro$. To this end, we extend the \NEW{false twin} relation $x\rthin y$ to
digraphs by setting $x\rthin y$ iff $x$ and $y$ have the same in- and
out-neighbors. The quotient graph $G/\!\rthin$ is known as the
\emph{point-determining graph} of $G$.
\begin{cor}
  An oriented graph $\overrightarrow{G}$ is explained w.r.t.\ $\Rtd$ if and
  only if $\overrightarrow{G}$ is an oriented forest whose
  point-determining graph does not contain
  $\bullet\to\bullet\leftarrow\bullet$ as an induced subgraph.
\end{cor}
\begin{proof} 
  It suffices to note that $\Ro$-equivalent vertices are in the same
  $\rthin$-class and that there is a least resolved tree in which all
  members of a $\Ro$-class are siblings.
\end{proof}

\begin{figure}
  \begin{center}
  \includegraphics[width=0.75\textwidth]{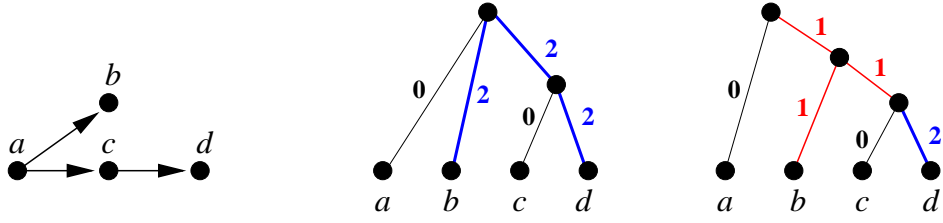}
  \end{center}
  \caption{Two distinct rooted trees with discrete $\Ro$ explaining an
    oriented graph on four vertices w.r.t.\ $\Rtd$.}
  \label{fig:directed}
\end{figure}

We note, finally, that the rooted trees with discrete $\Ro$ that explain
$\overrightarrow{G}$ w.r.t.\ $\Rtd$ are not unique, as exemplified in
Figure~\ref{fig:directed}.

\section{Concluding Remarks} 

The main result of this contribution is the characteriztion of the
exactly-$2$-relations, i.e.\ the graphs
$\textrm{nniPCG}(T,\lambda,2,2)$. They form a proper superset of the the
exact-2-leaf power graphs, which comprise only the disjoint unions of
cliques. Section~\ref{sect:general} suggests, however, that at least some
of the structure and techniques carry over to general values of
$k$. Several related problems are worth considering as well: in particular,
$\textrm{nniPCG}(T,\lambda,k,\infty)$ and $\textrm{nniPCG}(T,\lambda,1,k)$
are of interest as models for coarse grained models of evolutionary
distances.

The oriented version of the exactly-$2$-relation somewhat surprisingly, is
much more closely related to the oriented exactly-$1$-relation of
\cite{Hellmuth:17q} that to the undirected exactly-$2$-relation. There is
an alternative natural definition for a directed exactly-$2$-relation that
omit the condition that $\sum_{e\in \mathcal{P}(x,\lca{x,y})}\lambda(e)=0$.
Clearly, the resulting digraph are not oriented, i.e., they may contain
double edged. We suspect that their structure is more closely related to
the Fitch graph (directed at-least-$1$-relation) recently studied in
\cite{Geiss:17a}.

Regarding the analysis of rare-event data in phylogenetics the
characterization of the exactly-$2$-relation naturally leads to the edge
modification problem for block graphs and graphs whose R-thin quotient is a
block graph, respectively. Although these problems do not seem to have been
studied so far (see e.g.\ \cite[Tab.1]{Burzyn:06} and
\cite{Sritharan:16}). Since exactly-2-relation graphs are hereditary by
Lemma~\ref{lem:hered}, we suspect that edge modification problem for the
exactly-2-relation graphs can be handled in manner similar to closely
related edge modification problem for chordal graphs
\cite{Burzyn:06,Sritharan:16} or cluster editing \cite{Shamir:04}.

\subsection*{Acknowledgments} 
The authors gratefully acknowledge stimulating discussions with Marc
Hellmuth, Manuela Gei{\ss}, and Maribel Hern{\'a}ndez-Rosales on related
classes of \NEW{graphs} derived from labeled trees.

\bibliographystyle{elsarticle-num}
\bibliography{exact2relation}
 
\end{document}